\documentclass[10pt]{amsart}
\usepackage{amsmath}
\usepackage{amsfonts}
\usepackage{amssymb}
\usepackage{amsthm}
\usepackage{graphicx}
\usepackage{hyperref}

\usepackage{color}

\def \de {\partial}

\def \phi {\varphi}
\def \RN {\mathbb{R}^N}
\def \SN {S^{N-1}}

\def \R {\mathbb{R}}

\def \div {\text{div}}

\def \Gwr {\Gamma_{\omega, R}}
\def \Sw {\Sigma_\omega}

\newtheorem{theorem}{Theorem}[section]

\newtheorem{corollary}[theorem]{Corollary}

\theoremstyle{definition}

\numberwithin{equation}{section}

\begin{document}

\title{Integral formulas for hypersurfaces in cones and related questions}
\author[F. Pacella]{Filomena Pacella}
\address{Dipartimento di Matematica, Sapienza Universit\`a di Roma, P.le Aldo Moro 5 - 00185 Roma, Italy.
        }
 \email{pacella@mat.uniroma1.it}

\author[G. Tralli]{Giulio Tralli}
\address{Dipartimento di Matematica e Informatica, Universit\`a degli Studi di Ferrara, Via Machiavelli 30 - 44121 Ferrara, Italy.
         }
 \email{giulio.tralli@unife.it}



\begin{abstract}
We discuss the validity of Minkowski integral identities for hypersurfaces inside a cone, intersecting the boundary of the cone orthogonally. In doing so we correct a formula provided in \cite{CP}. Then we study rigidity results for constant mean curvature graphs proving the precise statement of a result given in \cite{PT1} and \cite{PT2}. Finally we provide an integral estimate for stable constant mean curvature hypersurfaces in cones.
\end{abstract}
\maketitle

\section{Introduction}\label{intro}

Let $\omega$ be a smooth open domain on the unit sphere $\SN$, $N\geq 2$, and let $\Sw$ be the open cone in $\RN$ with vertex at the origin $O$
$$\Sw=\{tx\,:\,x\in\omega,\,\,t\in(0,+\infty)\}.$$
We consider a smooth hypersurface $\Gamma\subset\Sw$ with $\de\Gamma\subset \de\Sw\smallsetminus\{O\}$, and denote by $\nu=\nu_x$ the (exterior) unit normal to $\Gamma$ for any $x\in\overline{\Gamma}$. Instead we denote by $n=n_x$ the exterior unit normal to the cone at any point $x\in \de\Sw\smallsetminus\{O\}$. We focus on the case when $\Gamma$ and $\de\Sw$ intersect orthogonally which means that we assume
\begin{equation}\label{ortho}
\left\langle \nu_x, n_x \right\rangle=0\qquad\forall\, x\in \de\Gamma.
\end{equation}
Let $\Gwr$ be the piece of a sphere centered at the vertex $O$ contained in the cone, i.e.
$$\Gwr=\de B_R\cap\Sw$$
where $B_R$ is the ball of radius $R>0$ centered at the origin. Then the orthogonality condition \eqref{ortho} is obviously satisfied for $\Gamma=\Gwr$.

\vskip 0.4cm

The main aim of our paper is to discuss the validity of integral identities of Hsiung-Minkowski type for hypersurfaces $\Gamma$ in the cone, satisfying the above orthogonality condition. \\
It is already known that the following first Minkowski formula holds true
\begin{equation}\label{Mink1}
\int_\Gamma{1-H\left\langle x,\nu_x\right\rangle}=0,
\end{equation}
where $H=H(x)$ is the mean curvature of $\Gamma$, see \cite[formula (4.8)]{RitRos} or \cite[Section 1]{CP}, see also \cite[page 858]{PT1} and \cite[formula 2.3]{PT2}. A short proof of \eqref{Mink1} will be provided in Section \ref{sec2}, for the reader convenience.\\
Denoting by $II^{\Sw}(\cdot,\cdot)$ the second fundamental form of $\de\Sw$ and by $II^{\Gamma}$ the one of $\Gamma$ with the associated second elementary symmetric function of its principal curvatures $\sigma_2$ (see Section \ref{sec2} for the definition) here we prove the following second Minkowski identity:

\begin{theorem}\label{thmintro}
Let $N\geq 3$. Under the standing assumptions for $\Sw$ and $\Gamma$, we have 
\begin{equation}\label{Mink2prima}
\int_\Gamma{H-\sigma_2\left\langle x,\nu\right\rangle}=\frac{-1}{(N-1)(N-2)}\int_{\de\Gamma}{II^{\Sw}\left(\nu-\left\langle x,\nu\right\rangle \frac{x}{|x|^2},\nu-\left\langle x,\nu\right\rangle \frac{x}{|x|^2}\right)  \left\langle x,\nu\right\rangle }.
\end{equation}
\end{theorem}

A second Minkowski formula in cones under the orthogonality condition is claimed in \cite[formula (7) in Proposition 1, case $k=2$]{CP} but it is wrong since the integral in the right-hand side of \eqref{Mink2prima} is missing. Actually this integral term is not zero in general when $N\geq 3$, not even in the case when the cone is convex (in which case is non-positive). In \cite{CP} also higher order Minkowski formulas are claimed and we believe that they all would need some additional term and additional arguments.\\
In Section \ref{sec2} we provide two proofs of Theorem \ref{thmintro}, the second one follows the approach of \cite{CP} with all details there missing.

\vskip 0.4cm

The wrong statement in \cite[Proposition 1]{CP} induced in turn a wrong rigidity theorem in our previous paper \cite[Theorem 1.1]{PT2}, see also \cite[Theorem 6.4]{PT1}, about constant mean curvature surfaces (CMC in short) $\Gamma$ which are starshaped with respect to the vertex of the (possibly non-convex) cone and intersect the cone in an orthogonal way.\\
Thus, in Section \ref{sec3} we prove the correct rigidity statement in Theorem \ref{prth1} which needs an integral gluing condition to hold.\\
Some connections between Theorem \ref{prth1} and isoperimetric inequality for almost convex cones proved by Baer-Figalli in \cite{BF} will be commented.

\vskip 0.4cm

Indeed the CMC surfaces intersecting the cone orthogonally are the critical points of the relative perimeter functional, under a volume constraint, while the isoperimetric inequality aims to characterize the minimizers of the perimeters. In \cite{LP} it has been proved that, in convex cones, the only sets in $\Sw$ which minimize the relative perimeter are the spherical sectors. A similar characterization was obtained in \cite{RitRos} for the stable critical points of the relative perimeter, in convex cones.\\
In Section \ref{sec4} we prove an inequality, Theorem \ref{lambdino}, for stable critical points which allows to relax the convexity assumption of \cite{RitRos}. It is interesting that it depends on the first Neumann eigenvalue of the Laplace-Beltrami operator on $\Gamma$. Recently, in \cite{IPW}, a connection between the stability of the spherical sector, as a critical point of the relative perimeter, and the first Neumann eigenvalue of the Laplace-Beltrami operator on $\omega\subset \SN$ was established. It would be interesting to further intestigate the role of these Neumann eigenvalues in the study of CMC surfaces in cones.

\subsection{Notations}

Let us further clarify here the assumptions and the notations used throughout the paper. The domain $\omega$ on the sphere $\SN$ is assumed to have a $C^2$-smooth boundary. Also, the embedded $(N-1)$-dimensional manifold $\Gamma$ is assumed to be relatively open, $C^3$-smooth in its relative interior, and $C^2$-smooth up to its (non-empty) relative boundary $\de\Gamma$ contained in $\de\Sw\smallsetminus\{O\}$; $\Gamma$ will be always bounded, connected, and orientable. Under these assumptions, there is a bounded open set $\Omega\subset\Sw$ (called \emph{sector-like domain} in \cite{PT1,PT2}) such that $\Gamma=\de\Omega\cap \Sw$: the choice of the unit normal $\nu$ is fixed as the exterior unit normal with respect to $\Omega$.

Having to deal with the two hypersurfaces $\Gamma$ and $\de\Sw$, we shall introduce the relevant geometric objects for our analysis. We denote by $\nabla$ the standard Levi-Civita connection in $\R^N$. Locally around any point in $\Gamma$, we can consider an orthonormal frame $\{e_1,\ldots,e_{N-1}\}$, and define (with a slight abuse of notations) by $\nabla^\Gamma$ the gradient with respect to the tangent directions to $\Gamma$ that is $\nabla^\Gamma u =\sum_{j=1}^{N-1} e_j(u) e_j$. Also, for $x\in \Gamma$, the second fundamental form $II_x=II^\Gamma_x$ of $\Gamma$ is the bilinear symmetric form on $T_x\Gamma\times T_x\Gamma$ which can be defined as
$$II^\Gamma(e_i,e_j)=\left\langle \nabla_{e_i}\nu,e_j\right\rangle,\qquad \mbox{for }i,j\in \{1,\ldots,N-1\}.$$
An important role in what follows is played by the mean curvature $H$ of $\Gamma$, and (in case $N\geq 3$) by the second elementary symmetric function $\sigma_2$ related to $II^\Gamma$: we recall here the definitions, namely
$$
H=\frac{1}{N-1}\sum_{j=1}^{N-1}II^\Gamma(e_j,e_j),\quad\mbox{ and }
$$
$$
\sigma_2=\frac{1}{(N-1)(N-2)}\left( (N-1)^2 H^2 - \sum_{i,j=1}^{N-1} (II^\Gamma(e_i,e_j))^2 \right).
$$
The analogous definitions, starting from the tangential derivatives of the normal direction to the cone $n$, are adopted for the second fundamental form $II^{\Sw}$ of the cone $\de\Sw$. The reader has to keep in mind that the radial direction $x$ (we decided to adopt the same notation for the position vector and the radial vector field with respect to $O$) is always a tangent direction for the cone, i.e.
\begin{equation}\label{conedef}
\left\langle x,n_x\right\rangle=0\qquad\mbox{ for any }x\in\de\Sw\smallsetminus\{O\}.
\end{equation}
Furthermore, the radial direction is also a direction of flatness in the sense that
\begin{equation}\label{flatx}
II_x^{\Sw}(x,\cdot)\equiv 0.
\end{equation}
One can easily check the previous identity by considering any tangential direction $e\in T_x\de\Sw$ (for arbitrary $x\in \de\Sw\smallsetminus\{O\}$) and computing $II_x^{\Sw}(x,e)=-\left\langle n_x,\nabla_e x \right\rangle=-\left\langle n,e\right\rangle=0$. We say that the cone is convex if
$$
II_x^{\Sw}\geq 0 \qquad\mbox{ for any }x\in\de\Sw\smallsetminus\{O\}
$$
in the sense of bilinear forms. Finally, we fix the notations for the first (non-trivial) eigenvalue for the Laplace-Beltrami operator on $\Gamma$ with homogeneous Neumann boundary conditions
$$
\lambda_1(\Gamma):=\inf\left\{\frac{\int_\Gamma |\nabla^\Gamma u|^2}{\int_\Gamma u^2}\,:\, u\in H^1(\Gamma)\,,\, u\neq 0\mbox{ such that }\int_\Gamma u=0\right\}.
$$
Without further explanations, we mention here that in all the previous integral formulas and also in what follows the relevant measures are always the induced surface measures of maximal (respectively $N-1$ and $N-2$) dimensions over $\Gamma$ and $\de\Gamma$.

\section{Minkowski integral identities}\label{sec2}

In this section we discuss the validity of integral identities of Hsiung-Minkowski type in the present setting ($\Gamma$ as above intersecting the cone $\Sw$ in an orthogonal way, thanks to \eqref{ortho}). As anticipated in the Introduction, the first formula of this kind reads as
\begin{equation}\label{Mink1dopo}
\int_\Gamma{1-H\left\langle x,\nu\right\rangle}=0.
\end{equation}
A proof of \eqref{Mink1dopo} can be quickly obtained and we reproduce it here for the convenience of the reader. Let us denote by $\div_\Gamma(F)$ the divergence of vector field which is tangent to $\Gamma$ and it is $C^1$-smooth up to $\de\Gamma$ (in our notations with the fixed orthonormal frame, this means $\div_\Gamma(F)=\sum_{j=1}^{N-1}\left\langle \nabla_{e_j} F,e_j\right\rangle$). If we consider the tangential part of the radial vector field
\begin{equation}\label{defF1}
F_1(x)=x-\left\langle x,\nu\right\rangle \nu,
\end{equation}
a straightforward computation shows that
$$
\div_\Gamma(F_1)=N-1-(N-1)H\left\langle x,\nu\right\rangle.
$$
Hence, by the divergence theorem we have
\begin{equation}\label{divuno}
\int_\Gamma{1-H\left\langle x,\nu\right\rangle} = \frac{1}{N-1}\int_\Gamma{\div_\Gamma(F_1)} = \int_{\de \Gamma} \left\langle x-\left\langle x,\nu\right\rangle \nu , n\right\rangle =0.
\end{equation}
We stress that, in the last equality, we used the fact that the conormal unit vector to $\de\Gamma$ (tangent to $\Gamma$) pointing outwards is equal to the exterior normal to the cone $n$ thanks to \eqref{ortho}, as well as the fact that we have the pointwise identities $\left\langle x,n\right\rangle=0=\left\langle \nu,n\right\rangle$ through $\de\Gamma$.\\
As far as a second Minkowski formula is concerned, the following holds

\begin{theorem}
Let $N\geq 3$, and assume $\Sw$ and $\Gamma$ as above. We have 
\begin{equation}\label{Mink2}
\int_\Gamma{H-\sigma_2\left\langle x,\nu\right\rangle}=\frac{-1}{(N-1)(N-2)}\int_{\de\Gamma}{II^{\Sw}\left(\nu-\left\langle x,\nu\right\rangle \frac{x}{|x|^2},\nu-\left\langle x,\nu\right\rangle \frac{x}{|x|^2}\right)  \left\langle x,\nu\right\rangle }.
\end{equation}
\end{theorem}
\begin{proof}
We provide two proofs of \eqref{Mink2}. For the first one we essentially collect known arguments used in \cite[formula (4.10)]{RitRos} or in \cite[Section 5 and formula (6.5)]{PT1}, and from these we derive \eqref{Mink2}. For the second proof we adopt the argument which was used in \cite[Proposition 1, case $k=2$]{CP} and led to the erroneous statement in \cite[formula (7), case $k=2$]{CP}.\\
\noindent{\it First proof. } Let $F_1$ be the vector field defined in \eqref{defF1}, and let $F_2$ be the one defined by
$$
F_2(x)=(N-1)H(x) F_1(x)-II_x^\Gamma(F_1,\cdot).
$$
This means that, if $\{e_1,\ldots,e_{N-1}\}$ describes locally our orthonormal frame, then
$$
F_2=\sum_{j,l=1}^{N-1}\left( II^\Gamma(e_l,e_l)\left\langle x,e_j\right\rangle - II^\Gamma(e_l,e_j)\left\langle x,e_l\right\rangle \right)e_j.
$$
A straightforward (yet known) computation shows that
\begin{equation}\label{divdue}
\div_\Gamma(F_2)= (N-1)(N-2) \left( H-\sigma_2 \left\langle x,\nu\right\rangle \right).
\end{equation}
Moreover, for any $x\in\de\Gamma$ we have
\begin{align*}
\left\langle F_2(x), n_x\right\rangle &= (N-1)H(x) \left\langle F_1(x), n_x\right\rangle - \left\langle \nabla_{F_1} \nu, n \right\rangle\\
&= (N-1)H(x) \left\langle x, n_x\right\rangle + \left\langle \nu,  \nabla_{F_1}  n \right\rangle  \\
&= \left\langle \nu,  \nabla_{F_1}  n \right\rangle,
\end{align*}
where in the second equality we used twice that the orthogonality condition \eqref{ortho} and in the third equality we used \eqref{conedef}. Keeping in mind the definition of second fundamental form for the cone $\de\Sw$ and the fact that $F_1$ is also tangent to the cone at the points in $\de\Gamma$ thanks to \eqref{conedef}-\eqref{ortho}, for any $x\in\de\Gamma$ we thus obtain
\begin{align}\label{duevoltebd}
\left\langle F_2(x), n_x\right\rangle&=II^{\Sw}(F_1,\nu)=II^{\Sw}(x-\left\langle x,\nu\right\rangle \nu,\nu)\\
&=II^{\Sw}(x,\nu)-\left\langle x,\nu\right\rangle II^{\Sw}(\nu,\nu)=-\left\langle x,\nu\right\rangle II^{\Sw}(\nu,\nu),\notag
\end{align}
where in the last equality we exploited \eqref{flatx}. Combining \eqref{duevoltebd} with \eqref{divdue} and keeping in mind that $n$ is also the conormal unit vector to $\de\Gamma$ pointing outwards, the divergence theorem yields
$$
(N-1)(N-2)\int_\Gamma{H-\sigma_2\left\langle x,\nu\right\rangle}=\int_\Gamma{ \div_\Gamma(F_2) }=\int_{\de\Gamma}{\left\langle F_2, n\right\rangle}= -\int_{\de\Gamma}{II^{\Sw}\left(\nu,\nu \right)  \left\langle x,\nu\right\rangle }.
$$
The previous identity, together with \eqref{flatx}, completes the proof of \eqref{Mink2}.\\
\noindent{\it Second proof. } Let us move $\Gamma$ in the normal direction $\nu$ in the following sense: for $t>0$ we define
$$
\Gamma_t=\left\{x_t=x+t\nu_x\,:\,x\in\Gamma\right\}
$$
We keep in mind that $\Gamma_0=\Gamma$, and the fact that for $t>0$ $\Gamma_t$ is not attached to $\de\Sw$. In particular, we denote by $n^t$ the conormal unit vector to $\de\Gamma_t$ pointing outwards ($n^0=n$ which is the normal to the cone at the points in $\de\Gamma$). We also notice that the unit normal direction to $\Gamma_t$ at the point $x_t=x+t\nu_x$ is given by $\nu_x$ (i.e. $\nu^t_{x_t}=\nu_x$). We can then define
$$
F_1^t(x_t)=x_t-\left\langle x_t, \nu_{x_t}^t\right\rangle \nu^t_{x_t}.
$$
The divergence theorem yields
\begin{equation}\label{daqui}
\int_{\Gamma_t} {\text{div}}_{\Gamma_t}\left( F_1^t \right)=\int_{\de\Gamma_t} \left\langle F_1^t, n^t \right\rangle.
\end{equation}
On one hand, as for \eqref{divuno} above, we have
$$
{\text{div}}_{\Gamma_t}\left( F_1^t \right) = (N-1)-(N-1)H_t\left\langle x_t,\nu^t \right\rangle.
$$
We can now integrate over $\Gamma_t$ and exploit a well-known expansion in $t$ of the mean curvature $H_t$ and of the area element (see, e.g., \cite[Section 3]{H} or \cite[page 877]{CP}) which provides the following asymptotics for $t\to 0^+$
\begin{align}\label{latodiv}
\int_{\Gamma_t}{\text{div}}_{\Gamma_t}\left( F_1^t \right)&=(N-1)\int_{\Gamma_t}{1-H_t\left\langle x_t,\nu^t \right\rangle}\\
&=(N-1)\left( \left( \int_{\Gamma}{1-H\left\langle x,\nu \right\rangle} \right) +  (N-2)t\left( \int_{\Gamma}{H-\sigma_2\left\langle x,\nu \right\rangle} \right) + o(t) \right).\notag
\end{align}
On the other hand, since we have both \eqref{conedef}-\eqref{ortho} at $t=0$, we deduce that
\begin{align}\label{latobordo}
\int_{\de\Gamma_t} \left\langle F_1^t, n^t \right\rangle &= t \frac{d}{dt}\left( \int_{\de\Gamma_t} \left\langle F_1^t, n^t \right\rangle \right)_{|t=0} + o(t)\\
&=t \int_{\de\Gamma}{ \frac{d}{dt}\left( \left\langle F_1^t, n^t \right\rangle \right)_{|t=0} } + o(t)\notag\\
&=t \int_{\de\Gamma}{ \left( \left\langle \nu ,n\right\rangle + \left\langle x, \frac{d}{dt}\left(n^t\right)_{|t=0} \right\rangle  \right)} + o(t)\notag\\
&=t \int_{\de\Gamma}{ \left\langle x, \frac{d}{dt}\left(n^t\right)_{|t=0} \right\rangle } + o(t)\qquad\mbox{ as $t\to 0^+$.}\notag
\end{align}
Combining \eqref{daqui}-\eqref{latodiv}-\eqref{latobordo}, we realize that
\begin{align*}
&(N-1)\left( \left( \int_{\Gamma}{1-H\left\langle x,\nu \right\rangle} \right) +  (N-2)t\left( \int_{\Gamma}{H-\sigma_2\left\langle x,\nu \right\rangle} \right) + o(t) \right) =\\
&= t \int_{\de\Gamma}{ \left\langle x, \frac{d}{dt}\left(n^t\right)_{|t=0} \right\rangle } + o(t) \qquad\mbox{ as $t\to 0^+$.}
\end{align*}
This is saying in particular that \eqref{Mink1} holds true (as we know), and also that
$$
(N-1)(N-2)\int_{\Gamma}{H-\sigma_2\left\langle x,\nu \right\rangle}=\int_{\de\Gamma}{ \left\langle x, \frac{d}{dt}\left(n^t\right)_{|t=0} \right\rangle }.
$$
This (second) proof of \eqref{Mink2} will be then complete once we show that
\begin{equation}\label{claim}
\left\langle x, \frac{d}{dt}\left(n_{x_t}^t\right)_{|t=0} \right\rangle = -\left\langle x,\nu\right\rangle II^{\Sw}_x(\nu,\nu)\quad\forall\, x\in\de\Gamma.
\end{equation}
To this aim, locally around any $x\in\de\Gamma$ we consider an orthonormal frame $\{e_1,\ldots,e_{N-2}\}$ for the tangent space to $\de\Gamma$. For each $j\in\{1,\ldots,N-2\}$, we have $(e_j)_x=\frac{d}{d\tau}\left(\gamma_x(\tau)\right)_{|\tau=0}$ (where $\gamma_x(\tau)\in\de\Gamma$ and $\gamma_x(0)=x$). Following the flux in the variable $t$ which maps $x\mapsto x_t$, the curve $\gamma_x$ is brought into
$$
\gamma^t_x(\tau)=\gamma_x(\tau)+t \nu_{\gamma_x(\tau)}.
$$
Hence, the related tangent vector field to $\de\Gamma_t$ becomes
$$
e_j^t=\frac{d}{d\tau}\left(\gamma^t_x(\tau)\right)_{|\tau=0}=e_j+t\nabla_{e_j}\nu.
$$
Now we notice that, for $t$ small enough, the set $\{e_1+t\nabla_{e_1}\nu,\ldots,e_{N-2}+t\nabla_{e_{N-2}}\nu\}$ describes a local basis for the tangent space of $\de\Gamma_t$. For such $t$, the direction of $n^t$ has to be perpendicular to ${\rm{span}}\{\nu,e_1+t\nabla_{e_1}\nu,\ldots,e_{N-2}+t\nabla_{e_{N-2}}\nu\}$ (these are $N-1$ linearly independent vectors in $\R^N$, and thus the direction of $n^t$ is uniquely determined). Since the vectors $e_j$'s were picked in an orthonormal way and since $n^0=n$, we can check directly that this implies that
$$
n^t=\frac{n-t\sum_{j=1}^{N-2}\left\langle \nabla_{e_j}\nu,n\right\rangle e_j + O(t^2)}{\left|n-t\sum_{j=1}^{N-2}\left\langle \nabla_{e_j}\nu,n\right\rangle e_j + O(t^2)\right|}\quad\mbox{ as }t\to 0^+.
$$
A direct computation shows then
\begin{align*}
\left\langle x, \frac{d}{dt}\left(n_{x_t}^t\right)_{|t=0} \right\rangle &= -\left\langle x, \sum_{j=1}^{N-2}\left\langle \nabla_{e_j}\nu,n\right\rangle e_j\right\rangle\\
&=-\sum_{j=1}^{N-2}\left\langle \nabla_{e_j}\nu,n\right\rangle \left\langle x,  e_j\right\rangle =- \left\langle \nabla_{F_1}\nu,n\right\rangle  \\
&=\left\langle \nu,\nabla_{F_1} n\right\rangle = II^{\Sw}(F_1,\nu) = -\left\langle x,\nu\right\rangle II^{\Sw}(\nu,\nu),
\end{align*}
where we used that $|n|=1$, together with \eqref{conedef} and \eqref{duevoltebd}. This establishes \eqref{claim}, and completes this second proof of \eqref{Mink2}.
\end{proof}

The reader can read the right hand side of \eqref{Mink2} as a correction term due to the fact that $\Gamma$ is an hypersurface with relative non-empty boundary. Since $\Gamma$ is hitting a cone and the gluing at the points in $\de\Gamma$ is also orthogonal, such a correction takes a precise geometric form involving the curvatures of $\de\Sw$. First and second Minkowski formulas with remainder terms have appeared also in other contexts in the literature. For example, in \cite{MT} the presence of such remainder terms is related to the underlying complex structure of the setting in which the closed hypersurfaces are embedded. 

\section{Rigidity statement: a corrected result}\label{sec3}

The integral identities established in the previous section enable us to deduce characterizations for constant mean curvature hypersurfaces $\Gamma$ which intersect the cone $\de\Sw$ orthogonally. In case the cone is convex such characterizations are well understood, and they have been established in \cite{RitRos, CP, PT1} through different methods. We are interested in relaxing such convexity assumptions. In particular, as anticipated in the Introduction, in the following theorem we correct a rigidity result for hypersurfaces $\Gamma$ which are also starshaped with respect to the vertex of the cone $O$ which was stated in \cite[Theorem 1.1]{PT2} (as a matter of fact, in that proof we exploited the uncorrect second Minkowski formula from \cite{CP} discussed in the previous section, see \cite[formula (2.4)]{PT2}). We recall that in our notations we say that $\Gamma$ is strictly starshaped if
$$
\left\langle x,\nu_x \right\rangle >0 \qquad\forall x\in \overline{\Gamma}. 
$$

\begin{theorem}\label{prth1}
Suppose that $\omega$ is strictly contained in an half-space of $\R^N$. Let $\Gamma\subset \Sw$ be a smooth $(N-1)$-dimensional manifold which is relatively open, bounded, orientable, connected and with $C^2$-smooth boundary $\de\Gamma$ contained in $\de\Sw\smallsetminus\{O\}$, and assume that $\Gamma$ and $\de\Sw$ intersect orthogonally at the points of $\de \Gamma$. Let us suppose that the mean curvature of $\Gamma$ is a constant $H$. If $\Gamma$ is strictly starshaped with respect to $O$ and if
\begin{equation}\label{ass}
\int_{\de\Gamma}{II^{\Sw}\left(\nu-\left\langle x,\nu\right\rangle \frac{x}{|x|^2},\nu-\left\langle x,\nu\right\rangle \frac{x}{|x|^2}\right)  \left\langle x,\nu\right\rangle }\geq 0,
\end{equation}
then $\Gamma=\Gamma_{\omega, \frac{1}{H}}$, i.e. $\Gamma$ is the (relative to $\Sw$)-boundary of a spherical sector.
\end{theorem}
\begin{proof}
It is obvious from \eqref{Mink1} that the constant $H$ is given by
$$
H=\frac{|\Gamma|}{\int_{\Gamma}\left\langle x,\nu\right\rangle}>0.
$$
Exploiting the fact that $H$ is constant and the two integral identities \eqref{Mink1}-\eqref{Mink2} we have also that
\begin{align*}
\int_{\Gamma} \left(H^2-\sigma_2\right)\left\langle x,\nu\right\rangle &= \int_{\Gamma} \left(H^2\left\langle x,\nu\right\rangle - H\right) + \int_{\Gamma} H-\sigma_2\left\langle x,\nu\right\rangle  \\
&= -H \int_{\Gamma} \left(1-H\left\langle x,\nu\right\rangle \right) + \int_{\Gamma} H-\sigma_2\left\langle x,\nu\right\rangle\\
&=\int_{\Gamma} H-\sigma_2\left\langle x,\nu\right\rangle\\
&=\frac{-1}{(N-1)(N-2)}\int_{\de\Gamma}{II^{\Sw}\left(\nu-\left\langle x,\nu\right\rangle \frac{x}{|x|^2},\nu-\left\langle x,\nu\right\rangle \frac{x}{|x|^2}\right)  \left\langle x,\nu\right\rangle }.
\end{align*}
By \eqref{ass} we have then
$$
\int_{\Gamma} \left(H^2-\sigma_2\right)\left\langle x,\nu\right\rangle \leq 0.
$$
On the other hand, we know that $\left\langle x,\nu\right\rangle>0$ and the arithmetic-geometric inequality ensures the validity of the pointwise inequality $\sigma_2\leq H^2$ through $\Gamma$. Hence
$$
0\leq \int_{\Gamma} \left(H^2-\sigma_2\right)\left\langle x,\nu\right\rangle \leq 0,
$$
and each of the previous inequalities is in fact an equality. In particular $\sigma_2= H^2$ for every point in $\Gamma$, which says that 
$$
II_x^{\Gamma}(\cdot,\cdot)=H \mathbb{I}_{N-1}\qquad\mbox{ for every }x\in\Gamma.
$$
Since $\Gamma$ is smooth and connected it is then a classical fact that such umbilicality property implies that $\Gamma$ is a portion of a sphere, i.e.
$$\exists p_0\in\RN\,\mbox{ such that }\Gamma=\de B_r(p_0)\cap\Sw.$$
We explicitly notice that the previous argument makes sense for every $N\geq 3$. In case $N=2$ the conclusion that $\Gamma$ is a piece of a circle still holds true by a simple and classical argument (in this case $H$ would be the curvature of the smooth curve $\Gamma$). We also stress that the condition \eqref{ass} is trivially satisfied in case $N=2$ (since the cones in $\R^2$ are of course locally flat).\\
In order to conclude the proof of the statement we have to prove that the point $p_0$ coincides with the vertex of the cone $O$. This is true thanks to \cite[Proposition 2.5]{PT2} since in our assumptions $\Sw$ cannot be an half-space.
\end{proof}

If we denote by $P_\omega(E)$ the relative perimeter of $E$ in $\Sw$, the constant mean curvature hypersurfaces $\Gamma$ which attach orthogonally to $\de\Sw$ arise as the relative boundaries of the critical points of $P_\omega(\cdot)$ under a constant costraint for the volume $|E|$. It has been proved in \cite{LP} (see also \cite{RitRos, FI, CRS}) that, in convex cones $\Sw$, the spherical sectors are the only sets contained in $\Sw$ which minimize $P_\omega$ under such a volume constraint. In \cite{BF} the authors established the validity of this isoperimetric property to cones that are almost convex, where \emph{almost} is intended with respect to the $C^{1,1}$-distance from the convex domain on the sphere. The proof in \cite[Theorem 1.2]{BF} is made in two steps: the first one consists of showing that in the almost convex cones the relative boundary of the minimizers are polar graphs (i.e. starshaped), and the second step establishes the radiality of the minimizers having a polar graph as relative boundary via a quantitative version of a Poincar\'e inequality in convex cones. This fact leads us to discuss the case of starshaped CMC hypersurfaces $\Gamma$ as in Theorem \ref{prth1} above. In our previous work \cite{PT2}, in order to provide an alternative approach to the Baer-Figalli result, we made use of the aforementioned \cite[Theorem 1.1]{PT2} in the proof of \cite[Theorem 3.3 and Corollary 3.4]{PT2}: as it is clear from the previous discussion, such approach does not work without assuming \eqref{ass}. In this respect, the assumption \eqref{ass} can be seen as an integral compatibility assumption which aims at relaxing the convexity of the cone (a similar phenomenon for the related overdetermined problem of Serrin type in cones has been noticed in \cite[Remark 3.1 and Proposition 3.2]{PT1}).

\section{Towards an eigenvalue criterion}\label{sec4}

In \cite{RitRos} the authors characterized the spherical sectors among the stable critical points of the relative perimeter functional in convex cones. On the other hand, in \cite{IPW} the stability property of the spherical sectors $\Gamma_{\omega,R}$ have been expressed in terms of the first Neumann eigenvalue of $\omega$. Having this in mind, in this section we establish a connection between smooth stable critical points and  $\lambda_1(\Gamma)$. We have the following

\begin{theorem}\label{lambdino}
Assume that $\Gamma$ is a smooth hypersurface as above, and it is the relative boundary of a stable critical point of $P_\omega$ under a constant volume constraint. Then we have
\begin{equation}\label{rel}
-\int_{\de\Gamma} II^{\Sw}\left(\nu-\left\langle x,\nu\right\rangle \frac{x}{|x|^2},\nu-\left\langle x,\nu\right\rangle \frac{x}{|x|^2}\right) \geq \lambda_1(\Gamma) \int_\Gamma \left| H x-\nu -  \frac{1}{|\Gamma|}\int_{\Gamma}(Hx-\nu)\right|^2,
\end{equation}
where $H$ is the (constant) mean curvature of $\Gamma$.
\end{theorem}
\begin{proof}
Define $F:\Gamma\rightarrow \R^N$ as
$$
F(x)=H x-\nu_x
$$
and
$$
F_0=\frac{1}{|\Gamma|}\int_\Gamma F(x) \in \R^N.
$$
We recall that, from the criticality condition, we have that $H$ is constant throughout $\Gamma$ and that $\Gamma$ meets $\de\Sw$ orthogonally at $\de\Gamma$. We notice that we can restrict to the case $N\geq 3$, since for $N=2$ the desired estimate is trivially satisfied. If $S^{N-1}=\{v\in\R^N\,:\, |v|=1\}$, for any $v\in S^{N-1}$ we denote
$$
u_v(x)=\left\langle F(x)-F_0, v\right\rangle.
$$
Being $F_0$ the average of $F(x)$ on $\Gamma$, we have
$$
\int_{\Gamma} u_v(x) =0.
$$
From the definition of $\lambda_1=\lambda_1(\Gamma)$ we infer
$$
\int_\Gamma |\nabla^\Gamma u_v (x)|^2  \geq \lambda_1 \int_\Gamma u_v^2(x) \qquad\mbox{ for all }v\in S^{N-1}.
$$
We now perform a trick due to Reilly in \cite[Main Lemma]{R} which consists in averaging with respect to $v$ (we denote by $d\sigma(v)$ the standard Hausdorff measure on the sphere $S^{N-1}$). We have
\begin{equation}\label{averagev}
\int_\Gamma\left( \frac{1}{|S^{N-1}|}\int_{S^{N-1}}|\nabla^\Gamma u_v(x)|^2 d\sigma(v)\right) \geq \lambda_1 \int_\Gamma\left( \frac{1}{|S^{N-1}|}\int_{S^{N-1}}u_v^2(x) d\sigma(v)\right).
\end{equation}
We can now compute, for every $x\in \Gamma$,
\begin{align}\label{mediauv}
&\frac{1}{|S^{N-1}|}\int_{S^{N-1}}u_v^2(x) d\sigma(v)=\frac{1}{|S^{N-1}|}\int_{S^{N-1}}\left\langle F(x)-F_0, v\right\rangle^2 d\sigma(v)\\
&=\frac{1}{|S^{N-1}|}\int_{\{|v|<1\}} \div_v\left( \left\langle F(x)-F_0, v\right\rangle (F(x)-F_0) \right)dv= \frac{|F(x)-F_0|^2}{N}.\notag
\end{align}
On the other hand, by considering an orthonormal frame $\{e_1,\ldots,e_{N-1}\}$ for the tangent space of $\Gamma$ locally around any point $x\in \Gamma$, we can write
$$
|\nabla^\Gamma u_v(x)|^2=\left|\sum_{j=1}^{N-1} \left\langle \nabla_{e_j} F(x), v\right\rangle  e_j\right|^2 =\sum_{j=1}^{N-1} \left\langle \nabla_{e_j} F(x), v\right\rangle^2.
$$
Hence, arguing as in \eqref{mediauv}, for every $x\in \Gamma$ we have
\begin{equation}\label{mediagraduv}
\frac{1}{|S^{N-1}|}\int_{S^{N-1}}|\nabla^\Gamma u_v(x)|^2 d\sigma(v)=\sum_{j=1}^{N-1} \frac{|\nabla_{e_j} F(x)|^2}{N}.
\end{equation}
Inserting \eqref{mediauv} and \eqref{mediagraduv} in \eqref{averagev}, we get
$$
\int_\Gamma \sum_{j=1}^{N-1} |\nabla_{e_j} F(x)|^2 d\sigma(x) \geq \lambda_1 \int_\Gamma |F(x)-F_0|^2 d\sigma(x).
$$
Recalling the definition of $F$ and the fact that $H$ is constant, we can compute
\begin{align*}
\sum_{j=1}^{N-1} |\nabla_{e_j} F(x)|^2&=\sum_{j=1}^{N-1}\left|H e_j - \nabla_{e_j} \nu\right|^2= \sum_{j=1}^{N-1}\left|\sum_{i=1}^{N-1} \left(H\delta_{ij}-\left\langle\nabla_{e_j}\nu, e_i\right\rangle\right) e_i\right|^2\\
&=\sum_{i,j=1}^{N-1}\left|H\delta_{ij}-\left\langle\nabla_{e_j}\nu, e_i\right\rangle\right|^2=(N-1)(N-2)\left(H^2-\sigma_2\right).
\end{align*}
This yields
\begin{equation}\label{yuppy}
(N-1)(N-2)\int_\Gamma (H^2-\sigma_2) \geq \lambda_1 \int_\Gamma |F(x)-F_0|^2.
\end{equation}
Comparing \eqref{yuppy} with the desired \eqref{rel}, it is clear that we need to show the following
\begin{equation}\label{claimstab}
(N-1)(N-2)\int_\Gamma (H^2-\sigma_2)+\int_{\de\Gamma} II^{\Sw}\left(\nu-\left\langle x,\nu\right\rangle \frac{x}{|x|^2},\nu-\left\langle x,\nu\right\rangle \frac{x}{|x|^2}\right) \leq 0.
\end{equation}
The validity of \eqref{claimstab} is exactly the ensured by the stability assumption for $\Gamma$. As a matter of fact, in \cite[formula right after (4.11)]{RitRos} it is proved that
$$
Q_\Gamma(\bar{u},\bar{u})= -(N-1)(N-2)\int_\Gamma (H^2-\sigma_2) - \int_{\de\Gamma} II^{\Sw}(\nu,\nu),
$$
where $Q_\Gamma(\bar{u},\bar{u})$ is the second variation of the relative perimeter functional $P_\omega$ under volume constraint computed at the critical point $\Gamma$ along the variation generated by the function $\bar{u}(x)=1-H\left\langle x,\nu\right\rangle$ (we recall that $\bar{u}$ generates an admissible variation which preserves the volume thanks to the first Minkowski formula \eqref{Mink1}). The proof of \eqref{claimstab} is thus complete since $Q_\Gamma(\bar{u},\bar{u})\geq 0$ from the stability assumption and then
\begin{align*}
(N-1)(N-2)\int_\Gamma (H^2-\sigma_2)\leq - \int_{\de\Gamma} II^{\Sw}(\nu,\nu) = -\int_{\de\Gamma} II^{\Sw}\left(\nu-\left\langle x,\nu\right\rangle \frac{x}{|x|^2},\nu-\left\langle x,\nu\right\rangle \frac{x}{|x|^2}\right),
\end{align*}
where we used \eqref{flatx}. This completes the proof of \eqref{claimstab}, and the combination of \eqref{claimstab} and \eqref{yuppy} finishes the proof of \eqref{rel}. 
\end{proof}

We remark that, in the case of convex cones, the left-hand side of \eqref{rel} is non-positive and $\lambda_1(\Gamma)\geq (N-1)$, which immediately implies that the vector $Hx-\nu$ is forced to be constant throughout $\Gamma$ and $\Gamma$ is forced to be a portion of a sphere. In this respect, the previous theorem provides a relaxed convexity assumption under which we have a characterization of (smooth) stable critical points. We can sum up this fact in the following corollary, for which the reader can see the similarities with Theorem \ref{prth1}.

\begin{corollary}
Suppose that $\omega$ is strictly contained in an half-space of $\R^N$. Assume $\Gamma$ is a smooth hypersurface in $\Sw$ as above, and $\Gamma$ is the relative boundary of a stable critical point of $P_\omega$ under a constant volume constraint. Suppose
$$
\int_{\de\Gamma} II^{\Sw}\left(\nu-\left\langle x,\nu\right\rangle \frac{x}{|x|^2},\nu-\left\langle x,\nu\right\rangle \frac{x}{|x|^2}\right)\geq 0.
$$
Then $\Gamma$ is the relative boundary of a spherical sector.
\end{corollary}

\subsection*{Acknowledgments}
\noindent We are very grateful to D. Ruiz and P. Sicbaldi to have pointed out that part of the proof of \cite[Proposition 1]{CP} was wrong, and for several useful discussions.

\bibliographystyle{amsplain}

\end{document}